\documentclass[11pt,reqno]{amsart}
\usepackage{amsmath,amssymb}
\makeatletter \oddsidemargin.9375in \evensidemargin \oddsidemargin
\begin{document}
\begin{center}
\thispagestyle{empty} \setcounter{page}{1} {\large\bf Jordan
$*-$homomorphisms between unital $C^*-$algebras \vskip.20in
{\Small \bf M. Eshaghi Gordji } \\[2mm]

{\footnotesize  $^1$Department of Mathematics,
Semnan University,\\ P. O. Box 35195-363, Semnan, Iran\\
[-1mm]e-mail: {\tt  madjid.eshaghi@gmail.com}}

}
\end{center}
\vskip 5mm \noindent{\footnotesize{\bf Abstract.} Let $A,B$ be two
unital $C^*-$algebras.  We prove that every almost unital almost
linear mapping $h:A\longrightarrow B$  which satisfies
$h(3^nuy+3^nyu) = h(3^nu)h(y)+h(y)h(3^nu)$  for all $u\in U(A)$, all
$y\in A$, and all $n = 0, 1, 2,...$, is a Jordan homomorphism. Also,
for a unital $C^*-$algebra $A$ of real rank zero, every almost
unital almost linear continuous mapping $h:A\longrightarrow B$ is a
Jordan homomorphism when $h(3^nuy+3^nyu) = h(3^nu)h(y)+h(y)h(3^nu)$
holds for all $u\in I_1(A_{sa})$, all $y\in A,$ and all $n = 0, 1,
2,...~~$. Furthermore, we investigate the Hyers--Ulam--Rassias
stability of Jordan $*-$homomorphisms between unital $C^*-$algebras
by using the fixed points methods.

 \vskip.10in

\vskip.10in
\newtheorem{df}{Definition}[section]
\newtheorem{rk}[df]{Remark}
\newtheorem{lem}[df]{Lemma}
\newtheorem{thm}[df]{Theorem}
\newtheorem{pro}[df]{Proposition}
\newtheorem{cor}[df]{Corollary}
\newtheorem{ex}[df]{Example}
\setcounter{section}{0} \numberwithin{equation}{section} \vskip
.2in
\begin{center}
\section{Introduction}
\end{center}

The stability of functional equations was first introduced  by S. M.
Ulam \cite{U} in 1940. More precisely, he proposed the following
problem: Given a group $G_1,$ a metric group $(G_2,d)$ and a
positive number $\epsilon$, does there exist a $\delta>0$ such that
if a function $f:G_1\longrightarrow G_2$ satisfies the inequality
$d(f(xy),f(x)f(y))<\delta$ for all $x,y\in G1,$ then there exists a
homomorphism $T:G_1\to G_2$ such that $d(f(x), T(x))<\epsilon$ for
all $x\in G_1?$ As mentioned above, when this problem has a
solution, we say that the homomorphisms from $G_1$ to $G_2$ are
stable. In 1941, D. H. Hyers \cite{H} gave a partial solution of
$Ulam^{,}s$ problem for the case of approximate additive mappings
under the assumption that $G_1$ and $G_2$ are Banach spaces. In
1978, Th. M. Rassias \cite{R1} generalized the theorem of Hyers by
considering the stability problem with unbounded Cauchy differences.
This phenomenon of stability that was introduced by Th. M. Rassias
\cite{R1} is called the Hyers--Ulam--Rassias stability. According to
Th. M. Rassias theorem:
\begin{thm}\label{t1} Let $f:{E}\longrightarrow{E'}$ be a mapping from
 a norm vector space ${E}$
into a Banach space ${E'}$ subject to the inequality
$$\|f(x+y)-f(x)-f(y)\|\leq \epsilon (\|x\|^p+\|y\|^p)$$
for all $x,y\in E,$ where $\epsilon$ and p are constants with
$\epsilon>0$ and $p<1.$ Then there exists a unique additive
mapping $T:{E}\longrightarrow{E'}$ such that
$$\|f(x)-T(x)\|\leq \frac{2\epsilon}{2-2^p}\|x\|^p$$ for all $x\in E.$
If $p<0$ then inequality $(1.3)$ holds for all $x,y\neq 0$, and
$(1.4)$ for $x\neq 0.$ Also, if the function $t\mapsto f(tx)$ from
$\Bbb R$ into $E'$ is continuous for each fixed $x\in E,$ then T
is linear.
\end{thm}
During the last decades several stability problems of functional
equations have been investigated by many mathematicians. A large
list of references concerning the stability of functional
equations can be found in \cite{Cz, H-I-R, J2}.\\
D.G.~Bourgin is the first mathematician dealing with the stability
of ring homomorphisms. The topic of approximate ring homomorphisms
was studied by a number of mathematicians, see \cite{B1, B-L-Z, Bo,
H-R,M-T-N, P2, P4, P-R, R2} and references therein.

Jun and Lee \cite{J-L} proved the following: Let X and Y be Banach
spaces. Denote by $\phi:X-\{0\} \times Y-\{0\}\to [0,\infty)$ a
function such that $\tilde{\phi}(x,y)=\sum_{n=0}^\infty
3^{-n}\phi(3^nx,3^ny)< \infty$  for all $x,y\in X - \{0\}.$ Suppose
that $f:X\longrightarrow Y$ is a mapping satisfying
$$\|2f(\frac{x+y}{2})-f(x)-f(y)\|\leq \phi(x,y)$$
 for all $x,y\in X - \{0\}.$ Then there exists a
unique additive mapping $T:X\longrightarrow Y$ such that
$$\|f(x)-f(0)-T(x)\|\leq \frac{1}{3}(\tilde{\phi}(x,-x)+\tilde{\phi}(-x,3x))$$
 for all $x\in X - \{0\}.$

Recently, C. Park and W. Park \cite{P-P} applied the Jun and Lee's
result to the Jensen's equation in Banach modules over a
$C^*-$algebra. B.E. Johnson (Theorem 7.2 of \cite{J}) also
investigated almost algebra
 $*-$homomorphisms between Banach  $*-$algebras: Suppose that U and B are
Banach  $*-$algebras which satisfy the conditions of (Theorem 3.1 of
\cite{J}). Then for each positive $\epsilon$ and K there is a
positive $\delta$ such that if $T\in L(U,B)$ with $\|T\|<K,
\|T^{\vee}\|<\delta$ and $\|T(x^*)^*-T(x)\|< \delta \|x\| (x\in U)$,
then there is a $*-$homomorphism $T':U\longrightarrow B$ with
$\|T-T'\| < \epsilon$. Here $L(U,B)$ is the space of bounded linear
maps from U into B, and $T^{\vee}(x, y) = T (xy)-T(x)T(y) (x,y \in
U)$. See \cite{J} for details. Throughout this paper, let A be a
unital $C^*-$algebra with  unit e, and B a unital $C^*-$algebra. Let
$U(A)$ be the set of unitary elements in A, $A_{sa}:= \{x\in A | x =
x^*\}$, and $I_1(A_{sa})=\{v \in A_{sa} | \|v\|=1, v\in Inv(A)\}$.
In this paper, we prove that every almost unital almost linear
mapping $h:A\longrightarrow B$ is a Jordan homomorphism when
$h(3^nuy+3^nyu) = h(3^nu)h(y)+h(y)h(3^nu)$ holds for all $u\in
U(A)$, all $y\in A$, and all $n = 0, 1, 2, . . .$ , and that for a
unital $C^*-$algebra A of real rank zero (see \cite{B-P}), every
almost unital almost linear continuous mapping $h:A\longrightarrow
B$ is a Jordan homomorphism when $h(3^nuy+3^nyu) =
h(3^nu)h(y)+h(y)h(3^nu)$ holds for all $u\in I_1(A_{sa})$, all $y\in
A,$ and all $n = 0, 1, 2, . . .$. Furthermore, we investigate the
Hyers--Ulam--Rassias stability of Jordan $*-$homomorphisms between
unital $C^*-$algebras by using the fixed pint methods.

Note that a unital  $C^*-$algebra is of  real rank zero, if the set
of invertible self--adjoint elements is dense in the set of
self--adjoint elements (see \cite{B-P}). We denote the algebric
center of algebra $A$ by $Z(A)$.

 \vskip 5mm
\section{Jordan $*-$homomorphisms on unital $C^*-$algebras}
By a following similar way as in \cite{P-B-A}, we obtain the next
theorem.

\begin{thm}\label{t1}

Let $f:A\to B$ be a mapping such that $f(0)=0$ and that
$$f(3^nuy+3^nyu) = f(3^nu)f(y)+f(y)f(3^nu)\eqno (2.1)$$  for all $u\in U(A)$, all
$y\in A$, and all $n = 0, 1, 2,...$. If  there exists a function
$\phi:(A-\{0\})^2\times A\to [0,\infty)$ such that
$\tilde{\phi}(x,y,z)=\sum_{n=0}^\infty 3^{-n}\phi(3^nx,3^ny,3^nz)<
\infty$ for all $x,y\in A - \{0\}$ and all $z\in A$ and that
$$\|2f(\frac{\mu x+\mu y}{2})-\mu f(x)-\mu f(y)+f(u^*)-f(u)^*\| \leq \phi(x,y,u), \eqno (2.2)$$
for all $\mu \in \Bbb T$ and all $x,y \in A, u\in (U(A)\cup \{0\}).$
If $\lim_n\frac{f(3^ne)}{3^n} \in U(B)\cap Z(B),$ then the mapping
$f:A\to B$ is a Jordan $*-$homomorphism.
\end{thm}
\begin{proof}Put $u=0, \mu=1$ in (2.2), it follows from Theorem 1 of \cite{J-L} that there exists a unique
additive mapping $T:A\to B$ such that
$$\|f(x)-T(x)\|\leq \frac{1}{3}(\tilde{\phi}(x,-x,0)+\tilde{\phi}(-x,3x,0))\eqno (2.3)$$
for all $x\in A-\{0\}.$ This mapping is given by
$$T(x)=\lim_n \frac{f(3^nx)}{3^n}$$
for all $x\in A.$ By the same reasoning as the proof of Theorem 1 of
\cite{P-B-A}, $T$ is $\Bbb C-$linear and $*-$preserving.  It follows
from (2.1) that
$$T(uy+yu)=\lim_n \frac{f(3^nuy+3^nyu)}{3^n}=\lim_n [\frac{f(3^nu)}{3^n}f(y)+f(y)\frac{f(3^nu)}{3^n}]
=T(u)f(y)+f(y)T(u)\eqno(2.4)$$ for all $u\in U(A)$, all $y\in A$.
Since $T$ is additive, then by (2.4), we have
$$3^nT(uy+yu)=T(u(3^ny)+(3^ny)u)=T(u)f(3^ny)+f(3^ny)T(u)$$
for all $u\in U(A)$ and all $y\in A$. Hence,
$$T(uy+yu)=\lim_n [T(u)\frac{f(3^ny)}{3^n}+\frac{f(3^ny)}{3^n}T(u)]
=T(u)T(y)+T(y)T(u)\eqno(2.5)$$ for all $u\in U(A)$ and all $y\in A$.
By the assumption, we have $$T(e)=\lim_n\frac{f(3^ne)}{3^n} \in
U(B)\cap Z(B)$$ hence, it follows by (2.4) and (2.5) that
$$2T(e)T(y)=T(e)T(y)+T(y)T(e)=T(ye+ey)=T(e)f(y)+f(y)T(e)=2T(e)f(y)$$
for all $y\in A.$ Since $T(e)$ is invertible, then $T(y)=f(y)$ for
all $y\in A.$ We have to show that $f$ is Jordan homomorphism. To
this end, let $x\in A.$ By Theorem 4.1.7 of \cite{K-R}, $x$ is a
finite linear combination of unitary elements, i.e.,
$x=\sum_{j=1}^nc_ju_j ~~(c_j\in \Bbb C, u_j \in U(A)),$ it follows
from (2.5) that
\begin{align*}f(xy+yx)&=T(xy+yx)=T(\sum_{j=1}^nc_ju_jy+\sum_{j=1}^nc_jyu_j)=\sum_{j=1}^nc_jT(u_jy+yu_j)\\
&=\sum_{j=1}^nc_j(T(u_jy)+T(yu_j))=\sum_{j=1}^nc_j(T(u_j)T(y)+T(y)T(u_j))\\
&=T(\sum_{j=1}^nc_ju_j)T(y)+T(y)T(\sum_{j=1}^nc_ju_j)=T(x)T(y)+T(y)T(x)\\
&=f(x)f(y)+f(y)f(x)
\end{align*}
for all $y\in A$. And this completes the proof of theorem.
\end{proof}
\begin{cor}\label{t2}
Let $p\in (0,1), \theta \in [0,\infty)$ be real numbers. Let $f:A\to
B$ be a mapping such that $f(0)=0$ and that
$$f(3^nuy+3^nyu) = f(3^nu)f(y)+f(y)f(3^nu)$$  for all $u\in U(A)$, all
$y\in A$, and all $n = 0, 1, 2,...$. Suppose that
$$\|2f(\frac{\mu x+\mu y}{2})-\mu f(x)-\mu f(y)+f(z^*)-f(z)^*\| \leq \theta (\|x\|^p+\|y\|^p+\|z\|^p)$$
for all $\mu \in \Bbb T$ and all $x,y,z \in A.$ If
$\lim_n\frac{f(3^ne)}{3^n} \in U(B)\cap Z(B),$ then the mapping
$f:A\to B$ is a Jordan $*-$homomorphism.

\end{cor}
\begin{proof}
Setting $\phi(x,y,z):=\theta(\|x\|^p+\|y\|^p+\|z\|^p)$ all $x,y,z
\in A.$ Then by Theorem 2.1 we get the desired result.
\end{proof}

\begin{thm}\label{t3}
Let $A$ be a $C^*-$algebra of real rank zero. Let $f:A\to B$ be a
continuous  mapping such that $f(0)=0$ and that
$$f(3^nuy+3^nyu) = f(3^nu)f(y)+f(y)f(3^nu)\eqno (2.6)$$  for all $u\in I_1(A_{sa})$, all
$y\in A$, and all $n = 0, 1, 2,...$. Suppose that  there exists a
function $\phi:(A-\{0\})^2\times A\to [0,\infty)$ satisfying (2.2)
and $\tilde{\phi}(x,y,z)< \infty$ for all $x,y\in A - \{0\}$ and all
$z\in A$. If $\lim_n\frac{f(3^ne)}{3^n} \in U(B)\cap Z(B),$ then the
mapping $f:A\to B$ is a Jordan $*-$homomorphism.
\end{thm}
\begin{proof}
By the same reasoning as the proof of Theorem 2.1, there exists a
unique involutive $\Bbb C-$linear mapping $T:A \to B$ satisfying
(2.3). It follows from (2.6) that
$$T(uy+yu)=\lim_n \frac{f(3^nuy+3^nyu)}{3^n}=\lim_n [\frac{f(3^nu)}{3^n}f(y)+f(y)\frac{f(3^nu)}{3^n}]
=T(u)f(y)+f(y)T(u)\eqno(2.7)$$ for all $u\in I_1(A_{sa})$, and all
$y\in A$. By additivity of $T$  and (2.7), we obtain that
$$3^nT(uy+yu)=T(u(3^ny)+(3^ny)u)=T(u)f(3^ny)+f(3^ny)T(u)$$
for all $u\in I_1(A_{sa})$ and all $y\in A$. Hence,
$$T(uy+yu)=\lim_n [T(u)\frac{f(3^ny)}{3^n}+\frac{f(3^ny)}{3^n}T(u)]
=T(u)T(y)+T(y)T(u)\eqno(2.8)$$ for all $u\in I_1(A_{sa})$ and all
$y\in A$. By the assumption, we have
$$T(e)=\lim_n\frac{f(3^ne)}{3^n} \in U(B)\cap Z(B).$$  Similar
to the proof of Theorem 2.1, it follows from (2.7) and (2.8) that
$T=f$ on $A$.  So $T$ is continuous. On the other hand $A$ is real
rank zero. On can easily show that $I_1(A_{sa})$ is dense in $\{x\in
A_{sa}: \|x\|=1\}$. Let $v\in \{x\in A_{sa}: \|x\|=1\}$. Then there
exists a sequence $\{z_n\}$ in $I_1(A_{sa})$ such that $\lim_n
z_n=v.$ Since $T$ is continuous, it follows from (2.8) that
\begin{align*}T(vy+yv)&=T(\lim_n (z_ny+yz_n))= \lim_n T(z_ny+yz_n)\\
&= \lim_n T(z_n)T(y)+\lim_n T(y)T(z_n)\\
&=
 T(\lim_n z_n)T(y)+ T(y)T(\lim_n z_n)\\
 &=T(v)T(y)+T(y)T(v) \hspace{7cm} (2.9)
\end{align*}
for all $y\in A$. Now, let $x\in A$. Then we have $x=x_1+ix_2,$
where $x_1:=\frac{x+x^*}{2}$ and $x_2=\frac{x-x^*}{2i}$ are
self--adjoint.

First consider  $x_1=0, x_2\neq 0.$ Since $T$ is $\Bbb C-$linear, it
follows from (2.9) that
\begin{align*}f(xy+yx)&=T(xy+yx)=T(ix_2y+y(ix_2))=T(i\|x_2\|\frac{x_2}{\|x_2\|}y+y(i\|x_2\|\frac{x_2}{\|x_2\|}))\\
&=i\|x_2\|T(\frac{x_2}{\|x_2\|}y+y\frac{x_2}{\|x_2\|})=i\|x_2\|[T(\frac{x_2}{\|x_2\|})T(y)+T(y)T(\frac{x_2}{\|x_2\|})]\\
&=T(i\|x_2\|\frac{x_2}{\|x_2\|})T(y)+T(y)T(i\|x_2\|\frac{x_2}{\|x_2\|})\\
&=T(ix_2)T(y)+T(y)T(ix_2)=T(x)T(y)+T(y)T(x)\\
&=f(x)f(y)+f(y)f(x)
\end{align*}
for all $y\in A$.

 If $x_2=0, x_1\neq 0,$ then by  (2.9), we have

\begin{align*}f(xy+yx)&=T(xy+yx)=T(x_1y+y(x_1))=T(\|x_1\|\frac{x_1}{\|x_1\|}y+y(\|x_1\|\frac{x_1}{\|x_1\|}))\\
&=\|x_1\|T(\frac{x_1}{\|x_1\|}y+y\frac{x_1}{\|x_1\|})=\|x_1\|[T(\frac{x_1}{\|x_1\|})T(y)+T(y)T(\frac{x_1}{\|x_1\|})]\\
&=T(\|x_1\|\frac{x_1}{\|x_1\|})T(y)+T(y)T(\|x_1\|\frac{x_1}{\|x_1\|})=T(x_1)T(y)+T(y)T(x_1)\\
&=T(x)T(y)+T(y)T(x)=f(x)f(y)+f(y)f(x)
\end{align*}
for all $y\in A$.

 Finally, consider the case that $x_1\neq 0, x_2\neq 0.$ Then it follows from (2.9) that
\begin{align*}f(xy+yx)&=T(xy+yx)=T(x_1y+ix_2y+yx_1+y(ix_2))\\
&=T(\|x_1\|\frac{x_1}{\|x_1\|}y+y(\|x_1\|\frac{x_1}{\|x_1\|})
+T(i\|x_2\|\frac{x_2}{\|x_2\|}y+y(i\|x_2\|\frac{x_2}{\|x_2\|}))\\
&=\|x_1\|T(\frac{x_1}{\|x_1\|}y+y\frac{x_1}{\|x_1\|})+i\|x_2\|T(\frac{x_2}{\|x_2\|}y+y\frac{x_2}{\|x_2\|})\\
&=\|x_1\|[T(\frac{x_1}{\|x_1\|})T(y)+T(y)T(\frac{x_1}{\|x_1\|})]+
i\|x_2\|[T(\frac{x_2}{\|x_2\|})T(y)+T(y)T(\frac{x_2}{\|x_2\|})]\\
&=[T(\|x_1\|\frac{x_1}{\|x_1\|})+T(i\|x_2\|\frac{x_2}{\|x_2\|})]T(y)+T(y)[T(\|x_1\|\frac{x_1}{\|x_1\|})+
T(i\|x_2\|\frac{x_2}{\|x_2\|})]\\
&=[T(x_1)+T(ix_2)]T(y)+T(y)[T(x_1)+T(ix_2)]\\
&=T(x)T(y)+T(y)T(x)=f(x)f(y)+f(y)f(x)
\end{align*}
for all $y\in A$. Hence, $f(xy+yx)=f(x)f(y)+f(y)f(x)$ for all
$x,y\in A$, and $f$ is Jordan $*-$homomorphism.
\end{proof}

\begin{cor}\label{t2}
Let $A$ be a $C^*-$algebra of rank zero. Let $p\in (0,1), \theta \in
[0,\infty)$ be real numbers. Let $f:A\to B$ be a mapping such that
$f(0)=0$ and that
$$f(3^nuy+3^nyu) = f(3^nu)f(y)+f(y)f(3^nu)$$  for all $u\in I_1(A_{sa})$, all
$y\in A$, and all $n = 0, 1, 2,...$. Suppose that
$$\|2f(\frac{\mu x+\mu y}{2})-\mu f(x)-\mu f(y)+f(z^*)-f(z)^*\| \leq \theta (\|x\|^p+\|y\|^p+\|z\|^p)$$
for all $\mu \in \Bbb T$ and all $x,y,z \in A.$ If
$\lim_n\frac{f(3^ne)}{3^n} \in U(B)\cap Z(B),$ then the mapping
$f:A\to B$ is a Jordan $*-$homomorphism.

\end{cor}
\begin{proof}
Setting $\phi(x,y,z):=\theta(\|x\|^p+\|y\|^p+\|z\|^p)$ all $x,y,z
\in A.$ Then by Theorem 2.3 we get the desired result.
\end{proof}

 \vskip 5mm
\section{stability of Jordan $*-$homomorphisms; a fixed point approach}

 We  investigate  the generalized Hyers--Ulam--Rassias
stability of Jordan $*-$homomorphisms on unital $C^*-$algebras by
using the alternative fixed point.

Recently, C$\breve{a}$dariu and Radu  applied the fixed point method
to the investigation of the  functional equations. (see also
\cite{C-R1, C-R2, C-R3, P-R, Ra, Ru}). Before proceeding to the main
result of this section,  we will state the following theorem.

\begin{thm}\label{t2}(The alternative of fixed point \cite{C-R}).
Suppose that we are given a complete generalized metric space
$(\Omega,d)$ and a strictly contractive mapping
$T:\Omega\rightarrow\Omega$ with Lipschitz constant $L$. Then for
each given $x\in\Omega$, either\\

$d(T^m x, T^{m+1} x)=\infty~$ for all $m\geq0,$\\
 or other exists a natural number $m_{0}$ such that\\

 $d(T^m x, T^{m+1} x)<\infty ~$for all $m \geq m_{0};$ \\

  the sequence $\{T^m x\}$ is convergent to a fixed point $y^*$ of $~T$;\\

 $y^*$is the unique fixed point of $~T$ in the
set $~\Lambda=\{y\in\Omega:d(T^{m_{0}} x, y)<\infty\};$\\

$ d(y,y^*)\leq\frac{1}{1-L}d(y, Ty)$ for all $~y\in\Lambda.$
\end{thm}

\begin{thm}\label{t31}
Let $f:A\to B$ be a mapping with $f(0)=0$ for which there exists a
function $\phi:A^5\to [0,\infty)$ satisfying

$$\|f(\frac{\mu x+\mu y+\mu z}{3})+f(\frac{\mu x-2\mu y+\mu z}{3})+f(\frac{\mu x+\mu y-2\mu z}{3})-
\mu f(x)+f(uv+uv)$$ $$-f(v)f(u)-f(u)f(v)+f(w^*)-f(w)^*\|
\leq\phi(x,y,z,u,v,w),\eqno(3.1)$$ for all $\mu\in \Bbb T,$ and all
$x,y,z,u,v \in A, w\in U(A)\cup \{0\}.$ If  there exists an $L<1$
such that $\phi(x,y,z,u,v,w)\leq 3L
\phi(\frac{x}{3},\frac{y}{3},\frac{z}{3},\frac{u}{3},\frac{v}{3},\frac{w}{3})$
for all $x,y,z,u,v,w\in A,$
 then there exists a unique Jordan $*-$homomorphism $h:A\to
B$ such that
$$\|f(x)-h(x)\| \leq \frac{L}{1-L}\phi(x,0,0,0,0,0)\eqno (3.2)$$ for all $x \in A.$
\end{thm}
\begin{proof}
It follows from $\phi(x,y,z,u,v,w)\leq 3L
\phi(\frac{x}{3},\frac{y}{3},\frac{z}{3},\frac{u}{3},\frac{v}{3},\frac{w}{3})$
that
$$lim_j 3^{-j}\phi(3^jx,3^jy,3^jz,3^ju,3^jv,3^jw)=0 \eqno(3.3)$$ for all $x,y,z,u,v,w\in
A.$\\
Put $ y=z=w=u=0$ in (3.1) to obtain
$$\|3f(\frac{x}{3})-f(x)\|\leq \phi(x,0,0,0,0,0)\eqno (3.4)$$
for all $x\in A.$ Hence,
$$\|\frac{1}{3}f(3x)-f(x)\|\leq \frac{1}{3} \phi(3x,0,0,0,0,0)\leq L\phi(x,0,0,0,0,0)\eqno (3.5)$$
for all $x\in A.$\\
Consider the set $X:=\{g\mid g:A\to B\}$ and introduce the
generalized metric on X:
$$d(h,g):=inf\{C\in \Bbb R^+:\|g(x)-h(x)\|\leq C\phi(x,0,0,0,0,0) \forall x\in A\}.$$
It is easy to show that $(X,d)$ is complete. Now we define  the
linear mapping $J:X\to X$ by $$J(h)(x)=\frac{1}{3}h(3x)$$ for all
$x\in A$. By Theorem 3.1 of \cite{C-R}, $$d(J(g),J(h))\leq Ld(g,h)$$
for
all $g,h\in X.$\\
It follows from (2.5) that  $$d(f,J(f))\leq L.$$ By Theorem 3.1, $J$
has a unique fixed point in the set $X_1:=\{h\in X: d(f,h)<
\infty\}$. Let $h$ be the fixed point of $J$. $h$ is the unique
mapping with
$$h(3x)=3h(x)$$ for all $x\in A$ satisfying there exists $C\in
(0,\infty)$ such that
$$\|h(x)-f(x)\|\leq C\phi(x,0,0,0,0,0)$$ for all $x\in A$. On the other hand we
have $lim_n d(J^n(f),h)=0$. It follows that
$$lim_n\frac{1}{3^n}{f(3^nx)}=h(x)\eqno (3.6)$$
for all $x\in A$. It follows from $d(f,h)\leq
\frac{1}{1-L}d(f,J(f)),$ that $$d(f,h)\leq \frac{L}{1-L}.$$ This
implies the inequality (3.2). It follows from (3.1), (3.3) and (3.6)
that
\begin{align*}\|&3h(\frac{x+y+z}{3})+ h(\frac{x-2y+z}{3})+h(\frac{x+y-2z}{3})- h( x)\| \\
&=lim_n
\frac{1}{3^n}\|f(3^{n-1}(x+y+z))+f(3^{n-1}(x-2y+z))+f(3^{n-1}(x+y-2z))-f(3^nx)\|\\
&\leq lim_n\frac{1}{3^n}\phi(3^nx,3^ny,,3^nz,0,0,0)\\
&=0
\end{align*}
for all $x,y,z \in A.$ So $$h(\frac{x+y+z}{3})+ h(\frac{x-2y+z}{3})+
h(\frac{x+y-2z}{3})= h( x)$$ for all $x,y,z \in A.$ Put
$w=\frac{x+y+z}{3}, t=\frac{x-2y+z}{3}$ and $s=\frac{x+y-2z}{3}$ in
above equation, we get $h(w+t+s)=h(w)+h(t)+h(s)$ for all $w,t,s\in
A.$ Hence, $h$ is Cauchy additive.
 By putting $y=z=x$, $v=w=0$ in (2.1), we have
$$\|\mu f(\frac {3\mu x}{3})-\mu f( x)\|\leq \phi(x,x,,x,0,0,0)$$
for all $x\in A$. It follows that
$$\|h(\mu x)-\mu h(x)\|=lim_m
\frac{1}{3^m}\|f(\mu 3^m x)-\mu f(3^m x)\|\leq
lim_m\frac{1}{3^m}\phi(3^mx,3^mx,3^mx,0,0,0)=0$$ for all $\mu \in
\Bbb T$, and all $x\in A.$ One can show that the mapping $h:A\to B$
is $\Bbb C-$linear.  By putting $x=y=z=u=v=0$ in (2.1) it follows
that
\begin{align*}\|&h(w^*)-(h(w))^*\|\\
&=lim_m\|\frac{1}{3^{m}}f((3^{m}w)^*)-\frac{1}{3^{m}}(f(3^{m}w))^*\|\\
&\leq lim_m \frac{1}{3^{m}}\phi(0,0,0,0,0,3^mw)\\
&=0
\end{align*}
for all $w \in U(A).$ By the same reasoning as the proof of Theorem
2.1, we can show that $h:A\to B$ is $*-$preserving.

 Since $h$ is $\Bbb C-$linear, by putting
$x=y=z=w=0$ in (2.1) it follows that
\begin{align*}\|&h(uv+vu)-h(v)h(v)-h(v)h(u)\|\\
&=lim_m\|\frac{1}{9^{m}}f(9^{m}(uv+vu))-\frac{1}{9^{m}}[f(3^{m}u)f(3^{m}v)+f(3^{m}v)f(3^{m}u)]\|\\
&\leq lim_m \frac{1}{9^{m}}\phi(0,0,0,3^mu,3^mv,0)\leq lim_m
\frac{1}{3^{m}}\phi(0,0,0,3^mu,3^mv,0)\\ &=0
\end{align*}
for all $u,v \in A.$  Thus $h:A\to B$ is Jordan $*-$homomorphism
satisfying (3.2), as desired.

\end{proof}
We prove the following Hyers--Ulam--Rassias stability problem for
Jordan $*-$homomorphisms on unital $C^*-$algebras.

\begin{cor}\label{t4}
Let $p\in (0,1), \theta \in [0,\infty)$ be real numbers. Suppose
$f:A \to A$ satisfies $$ \|\|f(\frac{\mu x+\mu y+\mu
z}{3})+f(\frac{\mu x-2\mu y+\mu z}{3})+f(\frac{\mu x+\mu y-2\mu
z}{3})- \mu f(x)+f(uv+uv)$$ $$-f(v)f(u)-f(u)f(v)+f(w^*)-f(w)^*\|\|
\}\leq \theta(\|x\|^p+\|y\|^p+\|z\|^p+\|u\|^p+\|v\|^p+\|w\|^p),$$
for all $\mu \in \Bbb T$ and all $x,y,z,u,v \in A, w\in U(A)\cup
\{0\}.$  Then there exists a unique
 Jordan $*-$homomorphism
$h:A\to B$ such that such that
$$\|f(x)-h(x)\| \leq \frac{3^p\theta}{3-3^p}\|x\|^p$$
for all $x \in A.$
\end{cor}
\begin{proof}
Setting
$\phi(x,y):=\theta(\|x\|^p+\|y\|^p+\|z\|^p+\|u\|^p+\|v\|^p+\|w\|^p)$
all $x,y,z,u,v,w \in A.$ Then by $L=3^{p-1}$ in Theorem 3.2,  one
can prove the result.
\end{proof}

{\small

}
\end{document}